\documentclass[a4paper,11pt]{article}

\usepackage{amsmath}
\usepackage{graphicx}
\usepackage{cite}
\usepackage{amsfonts}
\usepackage{amssymb}
\usepackage{amsthm}
\theoremstyle{plain}
\newtheorem{theorem}{Theorem}

\newtheorem{proposition}[theorem]{Proposition}

\theoremstyle{definition}

\begin{document}

\title{Delayed\ Feedback\ Control\ near\ Hopf\ Bifurcation}
\author{Fatihcan M. Atay \\
\centerline{Max Planck Institute for Mathematics in the Sciences} \\
\centerline{Inselstr.~22, Leipzig 04103, Germany} \\ \bigskip
\centerline{\texttt{atay@member.ams.org}}
}
\date{Preprint. Final version in: \\
\textit{Discrete and Continuous Dynamical Systems--S} 1:197--205 (2008)}
\maketitle


\begin{abstract} The stability of functional differential equations under
delayed feedback is investigated near a Hopf bifurcation. Necessary
and sufficient conditions are derived for the stability of the
equilibrium solution using averaging theory. The results are used to
compare delayed versus undelayed feedback, as well as discrete
versus distributed delays. Conditions are obtained for which
delayed feedback with partial state information can yield stability
where undelayed feedback is ineffective. Furthermore, it is shown
that if the feedback is stabilizing (respectively, destabilizing),
then a discrete delay is locally the most stabilizing (resp.,
destabilizing) one among delay distributions having the same mean.
The result also holds globally if one considers delays that are symmetrically
distributed about their mean.
\end{abstract}

\section{Introduction}

We study the effect of the feedback function $f$ on the stability of the zero
solution of the functional differential equation%
\begin{equation}
\dot{x}(t)=Lx_{t}+\varepsilon g(x_{t};\varepsilon)+\varepsilon\kappa
f(x_{t};\varepsilon), \label{main}%
\end{equation}
where $x(t)\in\mathbb{R}^{n},$ $x_{t}\in\mathcal{C}\triangleq C([-\tau
,0],\mathbb{R}^{n})$, $x_{t}(\theta)=x(t+\theta)\in\mathbb{R}^{n},$ $\theta
\in[-\tau,0],$ $L:\mathcal{C}\rightarrow\mathbb{R}^{n}$ is linear,
$\varepsilon$ is a small real parameter,
$f,g\in \mathcal{C}\times\mathbb{R}\to\mathbb{R}^{n}$
have continuous second derivatives with respect to each of their arguments and satisfy $f(0;\varepsilon)=g(0;\varepsilon)=0$ for all
$\varepsilon$, and $\kappa\in\mathbb{R}\ $denotes the feedback gain. It is
assumed that the linear problem obtained by setting $\varepsilon=0$ has a pair
of complex conjugate characteristic values $\pm i\omega\neq0$, and all other characteristic values
have negative real parts. Equation (\ref{main}) arises in the study of a
Hopf bifurcation of an equilibrium solution, after rescaling the space
variable $x\rightarrow\varepsilon x$ and the bifurcation parameter
$\alpha\rightarrow\varepsilon\alpha$; see e.g. \cite{Chow-MalletParet77}. In
applications, the problem is related to the feedback control of oscillations,
or conversely, to the oscillatory instabilities arising from delayed feedback
(e.g.~\cite{Atay-IJC02,Atay-LNCIS02}).

The aim of the present paper is to obtain precise conditions under which a
delayed feedback action can stabilize or destabilize an equilibrium solution
near a Hopf bifurcation, in particular when not all the system variables are
available for feedback. Moreover, we are interested in the difference between
discrete and distributed delays in the feedback. Taking advantage of being
near a Hopf bifurcation, we use averaging theory in Section \ref{sec:stab} to
derive necessary and sufficient conditions for stability. The implications for
delayed versus instantaneous feedback are investigated in Section
\ref{sec:partial}. Section~\ref{sec:versus} is devoted to a discussion of
discrete versus distributed delays.

\section{Stability of the zero solution}

\label{sec:stab} We introduce some notation. For $\varepsilon=0$, we write
(\ref{main}) as%

\begin{equation}
\dot{x}(t)=Lx_{t}=\int_{-\tau}^{0}d\eta(\theta)x(t+\theta) , \label{lin}%
\end{equation}
where $\eta$ is an $n\times n$ matrix whose components are of bounded
variation on $[-\tau,0].$ By assumption, (\ref{lin}) has a pair of characteristic values
$\pm i\omega\neq0$. By rescaling time it can be assumed that $\omega=1$
without loss of generality. Assume all other characteristic values have negative real
parts. Let $\Phi\ $be an $n\times2$ matrix whose columns span the eigenspace
of (\ref{lin}) corresponding to the characteristic value $\pm i$. In particular, $\Phi$
can be chosen such that
\begin{equation}
\Phi(\theta)=\Phi(0)e^{J\theta},\quad\theta\in[-\tau,0] \label{Phi}%
\end{equation}
where%
\begin{equation}
J=\left[
\begin{array}
[c]{cc}%
0 & -1\\
1 & 0
\end{array}
\right]  . \label{J}%
\end{equation}
Similarly, let $\Psi$ denote an $n\times2$ matrix whose columns span the
eigenspace corresponding to $\pm i$ for the adjoint equation
\begin{equation}
\dot{z}(t)=-\int_{-\tau}^{0}d\eta^{\top}(\theta)z(t-\theta) \label{adj}%
\end{equation}
on the space $\mathcal{C}^{\ast}=C([0,\tau],\mathbb{R}^{n}).$ Let $F$ and $G$
be $n\times n$ matrices, with elements of bounded variation on $[-\tau,0],$
such that
\begin{align}
[ D_{1}f(0;0)]\phi &  =\int_{-\tau}^{0}dF(\theta)\phi(\theta)\label{F}\\
[ D_{1}g(0;0)]\phi &  =\int_{-\tau}^{0}dG(\theta)\phi(\theta). \label{G}%
\end{align}
We define the scalar functions
\begin{align}
\hat{f}_1(\theta) &  =\operatorname*{tr}\left(\Psi^{\top}(0)F(\theta)\Phi(0)\right),
\quad
\hat{f}_2(\theta)  =\operatorname*{tr}\left(\Psi^{\top}(0)F(\theta)\Phi(0)J\right),
\label{Fhat}\\
\hat{g}_1(\theta) &  =\operatorname*{tr}\left(\Psi^{\top}(0)G(\theta)\Phi(0)\right),
\quad
\hat{g}_2(\theta)  =\operatorname*{tr}\left(\Psi^{\top}(0)G(\theta)\Phi(0)J\right),
\label{Ghat}%
\end{align}
where ``tr" denotes the matrix trace,
and define the real numbers $q,p$ by%
\begin{equation}
q=\int_{-\tau}^{0}\cos\theta\,d\hat{g_1}(\theta) + \int_{-\tau}^{0}\sin\theta\,d\hat{g_2}(\theta)   \label{q}%
\end{equation}%
\begin{equation}
p=\int_{-\tau}^{0}\cos\theta\,d\hat{f_1}(\theta) + \int_{-\tau}^{0}\sin\theta\,d\hat{f_2}(\theta)  \label{p}%
\end{equation}
Then for sufficiently small $\varepsilon$, the stability of the zero solution
of (\ref{main}) is given by the following result.

\begin{theorem}
\label{thm:stab} Let $\kappa\in\mathbb{R}$. There exists $\varepsilon_{0}>0$
such that for $\varepsilon\in(0,\varepsilon_{0})$ the origin is asymptotically
stable (unstable) if
\[
q+\kappa p <0 \quad(>0).
\]
\end{theorem}

\begin{proof}
By assumption, the linear system (\ref{lin}) and the adjoint system
(\ref{adj}) each have a $2$-dimensional center subspace, which are spanned by
the columns of the matrices $\Phi$ and $\Psi$, respectively. The bilinear form%
\begin{equation}
(\psi,\varphi):=\psi^{\top}(0)\varphi(0)-\int_{-\tau}^{0}\int_{0}^{\theta}%
\psi^{\top}(\zeta-\theta)d\eta(\theta)\varphi(\zeta)\,d\zeta
,\,\label{bilinear-form}%
\end{equation}
where $\psi\in\mathcal{C}^{\ast}$ and $\varphi\in\mathcal{C}$, allows a
decomposition of the space $\mathcal{C}=C([-\tau,0),\mathbb{R}^{n})$
\cite{Hale66}. Accordingly, the solution $x_{t}$ of the perturbed equation
(\ref{main}) can be written as
\[
x_{t}=\Phi y(t)+\chi_{t},\quad y(t)=(\Psi,x_{t})\,
\]
for some $\chi_{t}\in\mathcal{C}$, where $y$ satisfies%
\[
\dot{y}(t)=Jy(t)+\varepsilon\Psi^{\top}(0)\left[  g(\Phi y(t)+\chi
_{t};\varepsilon)+\kappa f(\Phi y(t)+\chi_{t};\varepsilon)\right]  ,
\]
with $J$ is as defined in (\ref{J}). The change of variables $y=\exp(Jt)u$,
gives%
\begin{equation}
\dot{u}(t)=\varepsilon e^{-Jt}\Psi^{\top}(0)\left(  g(\Phi e^{Jt}u(t)+\chi
_{t};\varepsilon)+\kappa f(\Phi_{1}e^{Jt}u(t)+\chi_{t};\varepsilon)\right)
.\label{udot}%
\end{equation}
By averaging, one obtains the equation
\begin{equation}
\dot{u}=\varepsilon\bar{g}(u)+\varepsilon\kappa\bar{f}(u),\label{u-avg}%
\end{equation}
where%
\begin{align}
\bar{f}(u) &  =\lim_{T\rightarrow\infty}\frac{1}{T}\int_{0}^{T}e^{-Jt}%
\Psi^{\top}(0)f(\Phi e^{Jt}u;0)\,dt\label{fbar}\\
\bar{g}(u) &  =\lim_{T\rightarrow\infty}\frac{1}{T}\int_{0}^{T}e^{-Jt}%
\Psi^{\top}(0)g(\Phi e^{Jt}u;0)\,dt.\label{gbar}%
\end{align}
It follows by the assumptions on $f$ and $g$ that $\bar{f}(0)=\bar{g}(0)=0$;
thus the origin is an equilibrium point of the system of equations
(\ref{u-avg}). The linear variational equation about the origin is%
\begin{equation}
\dot{u}=\varepsilon(\bar{G}+\kappa\bar{F})u,\label{u-lin}%
\end{equation}
where the averaged matrices $\bar{F},\bar{G}\in\mathbb{R}^{2\times2}$ are
defined by%
\begin{align*}
\bar{F} &  =\lim_{T\rightarrow\infty}\frac{1}{T}\int_{0}^{T}e^{-Jt}\Psi^{\top
}(0)\int_{-\tau}^{0}dF(\theta)\Phi(\theta)e^{Jt}\,dt\\
\bar{G} &  =\lim_{T\rightarrow\infty}\frac{1}{T}\int_{0}^{T}e^{-Jt}\Psi^{\top
}(0)\int_{-\tau}^{0}dG(\theta)\Phi(\theta)e^{Jt}\,dt,
\end{align*}
with $F$ and $G$ given in (\ref{F})--(\ref{G}). Applying Lemma 1 in
\cite{Atay-PHYSD03} to $\bar{F}$ and $\bar{G}$, we obtain
\begin{align}
\bar{F} &  =\frac{1}{2}\operatorname*{tr}\left(  \Psi^{\top}(0)\int_{-\tau
}^{0}dF(\theta)\Phi(\theta)\right)  \cdot I-\frac{1}{2}\operatorname*{tr}%
\left(  J\Psi^{\top}(0)\int_{-\tau}^{0}dF(\theta)\Phi(\theta)\right)  \cdot
J\label{Fbar}\\
\bar{G} &  =\frac{1}{2}\operatorname*{tr}\left(  \Psi^{\top}(0)\int_{-\tau
}^{0}dG(\theta)\Phi(\theta)\right)  \cdot I-\frac{1}{2}\operatorname*{tr}%
\left(  J\Psi^{\top}(0)\int_{-\tau}^{0}dG(\theta)\Phi(\theta)\right)  \cdot
J. \label{Gbar}%
\end{align}
From (\ref{Phi}), (\ref{Fhat}), and (\ref{Fbar}), the real parts
of the eigenvalues of $\bar{F}$ are both equal to
\begin{align*}
& \frac{1}{2}\operatorname*{tr}\left(\Psi^{\top}(0)\int_{-\tau}^{0}dF(\theta)\Phi(\theta)\right)\\
= & \frac{1}{2}\operatorname*{tr}\left(\Psi^{\top}(0)\int_{-\tau}^{0}dF(\theta)\Phi(0)e^{J\theta}\right)\\
= & \frac{1}{2}\operatorname*{tr}\left(\Psi^{\top}(0)\int_{-\tau}^{0}dF(\theta)\Phi(0)(I\cos\theta+ J\sin\theta)\right)\\
= & \frac{p}{2}
\end{align*}
Similarly, the real parts of the eigenvalues of $\bar{G}$ are equal to $q/2$, so that the real parts of the eigenvalues of the matrix
$\bar{G}+\kappa\bar{F}$ in (\ref{u-lin}) are given by
$\frac{1}{2}(q+\kappa p)$.
If $(q+\kappa p)\neq0$, the averaging theorem implies that
there exists $\varepsilon_{0}>0$ and an almost periodic solution $x^{\ast
}(\varepsilon)$ of the original equation (\ref{main}) for each $\varepsilon
\in [0,\varepsilon_{0}]$, which has the same stability type as the zero
solution of (\ref{u-lin}) \cite{Hale66}. Furthermore, $x^{\ast}(0)=0,$ and
$x^{\ast}$ is unique in a neighborhood of $0\in\mathcal{C}$ and $\varepsilon
=0$. It follows that $x^{\ast}(\varepsilon)\equiv0$ for $0\leq\varepsilon
\leq\varepsilon_{0}$, since zero is an almost periodic solution of the
averaged equation (\ref{u-avg}) for all $\varepsilon$. The theorem is then
proved since the stability of the zero solution of (\ref{u-lin}) is determined
by the sign of $q+\kappa p$.
\end{proof}

On the basis of the above theorem, we say that the feedback is \textit{stabilizing} or
\textit{destabilizing} depending on whether $\kappa p$ is negative or
positive, respectively. In this sense, the quantity $p$
quantifies and compares the (de)stabilizing effect
of the various feedback schemes given by $f$.

\section{Delayed versus instantaneous feedback}

\label{sec:partial} We now consider the role of delays in feedback with
partial state information. For this purpose, we assume that the linearized
feedback $F$ given in (\ref{F}) has the form%
\begin{equation}
F(\theta)=Ch(\theta) \label{F2}%
\end{equation}
where $h:[-\tau,0]\rightarrow\mathbb{R}$ is a function of bounded variation
representing a scalar distribution of delays, and $C\in\mathbb{R}^{n\times n}$
is a structure matrix. The feedback gain $\kappa$ is set to 1, or
alternatively subsumed into the matrix $C$. The interesting case is when $C$
does not have full rank, for instance when some of the system variables are
not available for feedback.
In the following we fix the number $q$ given in (\ref{q}) by fixing $L$ and
$g$, and investigate the effect of the structure matrix $C$ and the delay
distribution $h$ on stability.

Let $\hat{C}=\Psi^{\top}(0)C\Phi(0)$. From (\ref{p}) and (\ref{F2}),
it is seen that
\begin{equation}
p=\alpha\operatorname{tr}(\hat{C})+\beta\operatorname{tr}%
(\hat{C}J),\label{Re-p}%
\end{equation}
where
\begin{equation}
\alpha=\int_{-\tau}^{0}\cos\theta\,dh(\theta),\quad\beta=\int_{-\tau}^{0}%
\sin\theta\,dh(\theta).\label{ab}%
\end{equation}
In the absence of delays, i.e., when $h(\theta)$ is a Heaviside step 
function at zero, one has $\alpha=1$ and $\beta=0$, yielding
$p=\operatorname{tr}(\hat{C})$ for undelayed feedback. Hence, delayed
feedback is more stabilizing than undelayed feedback if%
\[
\alpha\operatorname{tr}(\hat{C})+\beta\operatorname{tr}%
(\hat{C}J)<\operatorname{tr}(\hat{C})
\]
or equivalently, if%
\begin{equation}
(1-\alpha)\operatorname{tr}(\hat{C})>\beta\operatorname{tr}%
(J\hat{C})\label{sta}%
\end{equation}
Similarly, delayed feedback is more destabilizing if (\ref{sta}) holds with
the inequality reversed. Although in applications the delays are often viewed
as destabilizing factors, the condition (\ref{sta}) shows the role of delays
in inducing stability. In particular, if $\operatorname{tr}(\hat{C})=0$ then
instantaneous feedback has no effect on the stability of the zero solution.
This case occurs, for instance, if only some of the system variables are used
in the feedback. However, if $\beta\operatorname{tr}(\hat{C}J)\neq0$, then by
(\ref{sta}) delayed feedback of the same variables can stabilize or
destabilize the zero solution, depending on the values of $\alpha$ and $\beta$.

To illustrate with an example, consider the classical van der Pol oscillator
with linear feedback control
\begin{equation}
\ddot{y}+\varepsilon(y^{2}-1)\dot{y}+y=\varepsilon\int_{-\tau}^{0}\left[
c_{1}y(t+\theta)+c_{2}\dot{y}(t+\theta)\right]  \,dh(\theta),\quad
0<\varepsilon\ll1.\label{vander}%
\end{equation}
With $x=(y,\dot{y})$, the linear equation around the origin is%
\[
\dot{x}(t)=-Jx(t)+\varepsilon\left(
\begin{array}
[c]{cc}%
0 & 0\\
0 & 1
\end{array}
\right)  x(t)+\varepsilon\left(
\begin{array}
[c]{cc}%
0 & 0\\
c_{1} & c_{2}%
\end{array}
\right)  \int_{-\tau}^{0}x(t+\theta)\,dh(\theta).
\]
We have
\[
\hat{C}=C=\left(
\begin{array}
[c]{cc}%
0 & 0\\
c_{1} & c_{2}%
\end{array}
\right)  ,
\]
giving $p=\alpha c_{2}-\beta c_{1}$. If the feedback is
instantaneous (i.e. without delays), then $\alpha=1$ and $\beta=0$, so that
$p$ depends only on the velocity feedback $c_{2}$. In this
case, the origin cannot be stabilized if velocity information is not available
for feedback. By contrast, if the feedback is delayed, then using only
position information can yield stability provided $\beta c_{1}>1$, by
Theorem \ref{thm:stab}. Figure \ref{fig:2} shows the quenching of oscillations
and the stabilization of the origin when $c_{2}=0$ and $h$ represents a
discrete delay at $\tau=1$.

\begin{figure}[ptb]
\begin{center}
\includegraphics[scale=0.45,angle=-90]{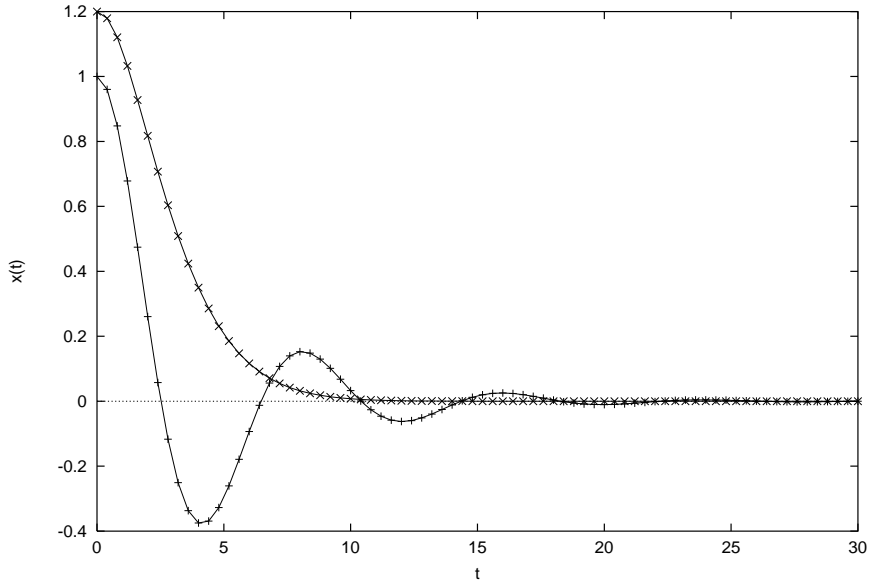}
\end{center}
\caption{Stabilization of the zero solution of the van der Pol oscillator
using position feedback, for feedback gain $c_{1}$ equal to $5$(+) and
$7.8$($\times$), and a discrete feedback delay at  $\tau=1$.
Other parameters are $c_{2}=0$, and $\varepsilon=0.1$.}%
\label{fig:2}%
\end{figure}

In closing this section we note that we have confined our discussion to linear
feedback. If nonlinear terms are added to the feedback, then it is possible to
further shape the system's behavior in addition to changing its stability. For
example, it has been shown that a limit cycle of desired amplitude can be
created \cite{Atay-IJC02}, or the period can be modified within certain
limitations \cite{Atay-LNCIS02}.

\section{Distributed versus discrete delays}

\label{sec:versus}
An interesting
question in stability investigations is how various distributions of delays about a given mean value affects
stability. In particular, one is interested in the difference between
distributed delays having the same mean delay
$\bar{\tau}=\int_{-\tau}^{0}\theta\,dh(\theta)$ and a discrete delay
at $\bar{\tau}$.
For example, Ref.~\cite{Boese89} studied the
stability of the Cushing equation with discrete and gamma-distributed delays.
In the context of a first-order system, it has been discussed that the
stability tends to improve with increasing variance of the delay distribution
\cite{Anderson91}, and it was~conjectured that a discrete delay at $\bar{\tau
}$ is more destabilizing than distributed delays having mean $\bar{\tau}$
\cite{Bernard01}. A further example involving coupled oscillators indeed
showed that increasing the variance of the delay distribution can enlarge the
stability region in the parameter domain \cite{PRL03}. We shall show that
these observations are true in a certain sense for Hopf instabilities of more
general systems. More precisely, when the delays act towards destabilizing the
system, the discrete delay is locally the most destabilizing one among delay
distributions having the same mean value. On the other hand, delays can also
stabilize an unstable equilibrium point, as seen in the previous section. In
this case the discrete delay is locally the most stabilizing delay distribution.

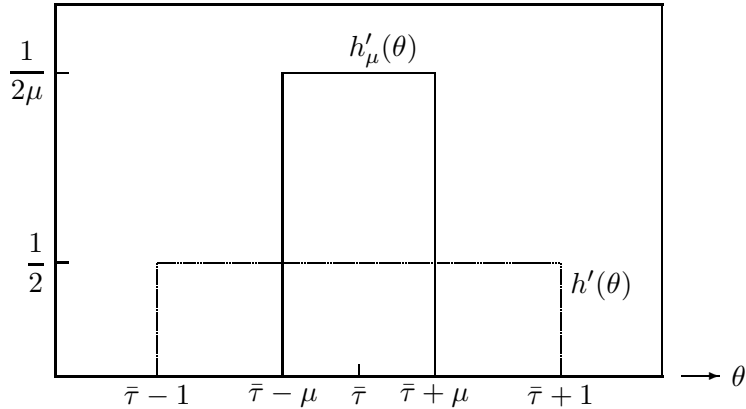
\begin{figure}
\begin{center}
\setlength{\unitlength}{0.18pt}
\begin{picture}(1600,900)(0,0)
\put(160,717){\makebox(0,0)[r]{$\dfrac{1}{2\mu}$}}
\put(180.0,717.0){\line(1,0){25}}
\put(160,320){\makebox(0,0)[r]{$\dfrac{1}{2}$}}
\put(180.0,320.0){\line(1,0){25}}
\put(1229,41){\makebox(0,0){$\bar{\tau}+1$}}
\put(967,41){\makebox(0,0){$\bar{\tau}+\mu$}}
\put(810,41){\makebox(0,0){$\bar{\tau}$}}
\put(810.0,82.0){\line(0,1){25}}
\put(652,41){\makebox(0,0){$\bar{\tau}-\mu$}}
\put(390,41){\makebox(0,0){$\bar{\tau}-1$}}
\put(1250,273){\makebox(0,0)[l]{$h^{\prime}(\theta)$}}
\put(789,765){\makebox(0,0)[l]{$h^{\prime}_{\mu} (\theta)$}}
\put(390,81.1){\dashbox{3}(839,238){}}
\put(390,82){\line(0,1){25}}
\put(1229,82){\line(0,1){25}}
\put(652,82){\framebox(315,635){}}
\put(180,82){\framebox(1259,778){}}
\put(1480,82){\vector(1,0){80}}
\put(1600,82){\makebox(0,0)[c]{$\theta$}}
\end{picture}
\caption{The densities $h^{\prime}_\mu$ corresponding to
uniformly distributed delays
about the mean value $\bar{\tau}$ and parametrized by $\mu$.}
\label{fig:unif-dist}
\end{center}
\end{figure}

To give a systematic study on the effect of delay distributions, we let
$\bar{\tau}$ be fixed and consider a family of distributions having mean value
$\bar{\tau}$. For this purpose, let $h$ be some reference distribution with
compact support satisfying\footnote{For simplicity,
here we view $h$ as a probability distribution,
that is, a monotone function satisfying (\ref{probability}).
The particular choice of $h$ will not be important
for the following discussion.}
\begin{equation}
\int_{-\infty}^{\infty}dh(\theta)=1,\quad\int_{-\infty}^{\infty}%
\theta\,dh(\theta)=\bar{\tau},\quad\int_{-\infty}^{\infty}(\theta-\bar{\tau
})^{2}\,dh(\theta)=\sigma^{2}\text{,}%
\label{probability}
\end{equation}
and define a family of distributions parametrized by $\mu>0,$%
\begin{equation}
h_{\mu}(\theta)=h(\bar{\tau}+(\theta-\bar{\tau})/\mu).\label{h-mu}%
\end{equation}
(Figure~\ref{fig:unif-dist} depicts the rescaling (\ref{h-mu})
for the case of uniformly distributed delays.)
It is easy to check that%
\[
\int_{-\infty}^{\infty}dh_{\mu}(\theta)=1,\quad\int_{-\infty}^{\infty}%
\theta\,dh_{\mu}(\theta)=\bar{\tau},\quad\int_{-\infty}^{\infty}(\theta
-\bar{\tau})^{2}\,dh_{\mu}(\theta)=\mu^{2}\sigma^{2}\text{.}%
\]
Hence, the family $h_{\mu}$ provides a natural way to change the
variance $\mu^{2}\sigma^{2}$ of the distribution while keeping the mean value
fixed. We denote the corresponding values in (\ref{ab}) as%
\begin{equation}
\alpha_{\mu}=\int_{-\infty}^{\infty}\cos\theta\,dh_{\mu}(\theta),\quad
\beta_{\mu}=\int_{-\infty}^{\infty}\sin\theta\,dh_{\mu}(\theta).\label{ab-mu}%
\end{equation}
We also define%
\begin{equation}
\alpha_{0}=\cos\bar{\tau},\text{\quad}\beta_{0}=\sin
\bar{\tau},\label{ab-0}%
\end{equation}
i.e., the values of $\alpha,\beta$ for a discrete delay at $\bar{\tau}$, which
is equivalent to letting $h_{0}$ be a Heaviside step function at $\bar{\tau}$. 
The question is, for a fixed structure matrix $\hat{C}$, how the quantity
\begin{equation}
p_{\mu}=\alpha_{\mu}\operatorname{tr}(\hat{C})+\beta_{\mu
}\operatorname{tr}(J\hat{C})\label{Re-p-mu}%
\end{equation}
changes as $\mu$ is varied.

Before considering the general case, it is instructive to first look
at the specific example
of uniformly distributed delays shown in Figure~\ref{fig:unif-dist}.
Here one has
\[
\alpha_\mu
= \frac{1}{\mu} \sin\mu \cos\bar{\tau},
\quad
\beta_\mu
= \frac{1}{\mu} \sin\mu \sin\bar{\tau},
\]
which yields
\begin{align}
p_\mu & = \frac{\sin\mu}{\mu}
\left( \operatorname{tr}(\hat{C})\cos\bar{\tau}
+\operatorname{tr}(J\hat{C})\sin\bar{\tau} \right) \nonumber \\
& = \frac{\sin\mu}{\mu} p_0.
\label{pmu}
\end{align}
It is thus seen that the dependence of $p_\mu$ on the parameter $\mu$
is not monotone. In fact, by Theorem~\ref{thm:stab}, the sign changes
of $p_\mu$ for varying values of $\mu$ indicates that stability switches can occur as the variance of the delay distribution is changed.
Nevertheless, since $|\sin\mu| < |\mu|$ for all $\mu>0$, one has
$|p_0| > |p_\mu|$, which shows that a discrete delay has a stronger
effect on stability than all uniformly distributed delays
having the same mean value.
Moreover, there are certain values of the distribution variance
(given by $\sin\mu=0$, $\mu>0$) for which the feedback has no effect on stability.
This last property is purely an effect of the delay variance and is independent of the choice of the structure matrix $C$.

The extremal property of discrete delays observed above can be
extended to more general delay distributions.
We first give a local characterization.

\begin{proposition}%
\label{thm:local}
\[
\left.  \frac{\partial p_{\mu}}{\partial\mu} \right|
_{\mu=0}=0\text{,\quad and\quad}\left.
\frac{\partial^{2}p_{\mu}}{\partial\mu^{2}%
}\right|  _{\mu=0}=-\sigma^{2}p_{0}
\]
\end{proposition}

\begin{proof}
By a change of variables in (\ref{ab-mu}) and using (\ref{h-mu}), we have
\begin{equation}
\alpha_{\mu}=\int_{-\infty}^{\infty}\cos\theta\,dh_{\mu}(\theta)=\int
_{-\infty}^{\infty}\cos(\bar{\tau}+\mu(s-\bar{\tau}))\,dh(s).
\label{alpha_mu}
\end{equation}
Using (\ref{ab-0}) it is seen that $\alpha_{\mu}$ and $\beta_{\mu}$ are smooth
functions of $\mu$ on $\mathbb{R}$. Differentiating under the integral gives%
\[
\frac{\partial\alpha_{\mu}}{\partial\mu}=-\int_{-\infty}^{\infty}\sin
(\bar{\tau}+\mu(s-\bar{\tau}))(s-\bar{\tau})\,dh(s).
\]
Thus $\partial\alpha_{\mu}/\partial\mu|_{\mu=0}=0$. Similarly, $\partial
\beta_{\mu}/\partial\mu|_{\mu=0}=0.$ On the other hand,%
\begin{align*}
\left.  \frac{\partial^{2}\alpha_{\mu}}{\partial\mu^{2}}\right|  _{\mu=0} &
=-\sigma^{2}\cos\bar{\tau}=-\sigma^{2}\alpha_{0},\\
\left.  \frac{\partial^{2}\beta_{\mu}}{\partial\mu^{2}}\right|  _{\mu=0} &
=-\sigma^{2}\sin\bar{\tau}=-\sigma^{2}\beta_{0}.
\end{align*}
Using in (\ref{Re-p-mu}), we obtain the conclusion.
\end{proof}

By the above result, if $p_{0}$ is nonzero, then it is a
local extremum for $p_{\mu}$. This shows that discrete
delays are indeed special in a certain sense. Thus, if the delayed feedback is
a destabilizing one ($p_{\mu}>0$), then a discrete delay is
locally the most destabilizing delay distribution, and increasing the variance
of the distribution reduces $p_{\mu}$. This confirms the
observations of \cite{Anderson91,Bernard01,PRL03} and shows that it is
generally true near a Hopf instability. However, as noted above, delays can
also have a stabilizing effect ($p_{\mu}<0$), in which case
a discrete delay is locally the most stabilizing one, and increasing the
variance of the distribution can yield instability. In both cases, increasing
the variance of the distribution locally about a discrete delay reduces the
effect of delays in the feedback.

For symmetrically distributed delays, we can also give a global
characterization of the extremal property of discrete delays.

\begin{proposition}
\label{thm:global}
For delay distributions that are symmetrically distributed about
their mean value\footnote{That is,
$h^{\prime}(\bar{\tau}+\theta)=h^{\prime}(\bar{\tau}-\theta)$,
where the derivative exists a.~e.~by assumption.},
 $|p_\mu| \le |p_0|$ for all $\mu > 0$.
\end{proposition}

\begin{proof}
Expanding the cosine term in (\ref{alpha_mu}),
\[
\alpha_\mu =
\cos\bar{\tau}\int^{\infty}_{-\infty}\cos(\mu(s-\bar{\tau}))\,dh(s) -
\sin\bar{\tau}\int^{\infty}_{-\infty}\sin(\mu(s-\bar{\tau}))\,dh(s).
\]
The second integral vanishes because the distribution is symmetric about
$\bar{\tau}$ and sine is an odd function. Thus by (\ref{ab-0}),
$\alpha_\mu =
\alpha_0\int^{\infty}_{-\infty}\cos(\mu(s-\bar{\tau}))\,dh(s)$,
and similarly
$\beta_\mu =
\beta_0\int^{\infty}_{-\infty}\cos(\mu(s-\bar{\tau}))\,dh(s)$.
Hence, from (\ref{Re-p-mu})),
\[
p_\mu = p_0 \int^{\infty}_{-\infty}\cos(\mu(s-\bar{\tau}))\,dh(s).
\]
Then the estimate
\[
|p_\mu| \le |p_0| \int^{\infty}_{-\infty}|\cos(\mu(s-\bar{\tau}))|\,
dh(s)\le |p_0| \int^{\infty}_{-\infty}\,dh(s) = |p_0|
\]
follows.
\end{proof}

Finally we note that $p_0$ depends only on the mean delay and not on
the distribution $h$. Hence, the extremal properties of discrete delays
given in Propositions \ref{thm:local} and \ref{thm:global} are independent of the particular choice of the reference distribution $h$.

\end{document}